\documentclass[11pt, reqno]{amsart} 
\usepackage{ amsmath, amsthm, amscd, amsfonts, amssymb, graphicx, color, enumerate}
\usepackage[bookmarksnumbered, colorlinks, plainpages]{hyperref}

\definecolor{TITLE}{rgb}{0.0,0.0,1.0}
\definecolor{AUTHOR1}{rgb}{0.00,0.59,0.00}
\definecolor{AUTHOR2}{rgb}{0.50,0.00,1.00}
\definecolor{SECTION}{rgb}{0.50,0.00,1.00}
\definecolor{FOOTTITLE}{rgb}{0.00,0.50,0.75}
\definecolor{THM}{rgb}{0.7,0.3,0.3}
\definecolor{SEC}{rgb}{0.6,0.1,.5}

\makeatletter
\@namedef{subjclassname@2010}{%
\textup{2010} Mathematics Subject Classification}
\makeatother
\newtheorem{theorem}{{\color{THM} Theorem}}[section]

\newtheorem{lemma}[theorem]{{\color{THM}Lemma}}
\newtheorem{proposition}[theorem]{{\color{THM}Proposition}}
\newtheorem{corollary}[theorem]{{\color{THM}Corollary}}
\theoremstyle{definition}

\newtheorem{example}[theorem]{{\color{THM}Example}}

\numberwithin{equation}{section}

\numberwithin{equation}{section}

\def\speaker{$ ^{*} $\protect\footnotetext{$ ^{*} $\lowercase{Corresponding Author.}}}
\begin{document}

\title {On the Generalized Quotient Integrals on Homogenous Spaces }

\author{ T. Derikvand, R. A. Kamyabi-Gol\speaker$ $ and M. Janfada}

\address{International Campus, Faculty of Mathematic Sciences, Ferdowsi University of Mashhad, Mashhad, Iran} 
\email{derikvand@miau.ac.ir}

\address{ Department of Pure Mathematics and Centre of Excellence
in Analysis on Algebraic Structures (CEAAS), Ferdowsi University of
Mashhad, P.O. Box 1159, Mashhad 91775, Iran}
\email{kamyabi@um.ac.ir}

\address{ Department of Pure Mathematics, Ferdowsi University of
Mashhad, P.O. Box 1159, Mashhad 91775, Iran}
\email{Janfada@um.ac.ir}

\begin{abstract} 
A generalization of the quotient integral formula is presented and some of its properties are investigated. Also the relations between two function spaces related to the spacial homogeneous spaces are derived by using general quotient integral formula. Finally our results are supported by some examples. \end{abstract}


\maketitle
\section{INTRODUCTION}
In $ 1917 $ Johann Radon in \cite{10} showed that a differentiable function on $ \mathbb{R}^2 $ or $ \mathbb{R}^3 $ could be recovered explicitly from its integrals over lines or planes, respectively. Nowadays, this reconstruction problem has many applications in different areas. For example, if the Euclidian group $E(2)$ acts transitively on $\mathbb{R}^2$, then the isotropy subgroup of origin is the orthogonal group $O(2)$. In that sequel, the homogeneous space $ E(2)/O(2) $ provides definition of X-ray transform that is used  in many areas including computerized tomography, magnetic resonance imaging, positron emission tomography, radio astronomy, crystallographic texture analysis, etc. See \cite{1, 2}. This shows that the study of function spaces related to homogenous spaces is useful, so that the Radon transform and its applications has been studied by many authors; see, for example, \cite{4, 7, 9, 10}. S. Helgason \cite{5} studied the Radon transform in the more general framework of homogeneous spaces for a topological group $ G $.\\
Let $G$ be a locally compact group, $H$ and $K$ be two closed subgroups of $G$, $L:=H \cap K$ and also let $X:=G/K$ and $Y:=G/H$ denote two left coset spaces of $G$. Assume that $d\sigma$=$d(kL)$ and $d\eta=d(hL)$ are two $K-$invariant and $H-$invariant Radon measures on $K/L$ and $H/L$, respectively. The Radon transform $R_{K, H}:C_{c}(G/K)$$ \rightarrow $$C(G/H)$ and its dual $R^{*}_{K, H}:C_{c}(G/H)$$ \rightarrow $$C(G/K)$ 
were introduced by S. Helgason in 1966 defined by 
\begin{align}
R_{K, H}f(xH)=\int_{H/L}{f(xhK)d(hL)}\qquad (f\in C_{c}(G/K)),\label{eq11}
\end{align}
and
\begin{align}
R^{*}_{K, H}\varphi(xK)=\int_{K/L}{\varphi(xkH)d(kL)}\qquad (\varphi \in C_{c}(G/H)).\label{eq12}
\end{align}

Regarding the abstract setting of the Radon transform on function spaces related to homogeneous spaces, we will generalize some important known results from locally compact group to homogeneous spaces. The outline of the rest of this paper is as follows: In section 2 we mention the preliminaries including a brief summary on homogeneous spaces and the (quasi) invariant measures on them. A generalization for the quotient integral formula is presented in the third section. 
Finally, the results are supported by some examples.

\section{PRELIMINARIES}
In the sequel, $H$ is a closed subgroup of a locally compact group $G$ and $ dx $, $ dh $ are the left Haar measures on $ G $ and $ H $, respectively. 
The modular function $ \triangle_{G} $ is a continuous homomorphism from $G$ into the multiplicative group $ \mathbb{R}^{+} $. Furthermore,
\[\int_{G} f(y)dy=\triangle_{G}(x)\int_{G} f(yx)dy\]
where $ f\in C_{c}(G) $, the space of continuous functions on $ G $ with compact support, and $ x\in G $ (Proposition 2.24 of \cite{3}).
A locally compact group $ G $ is called unimodular if $ \triangle_{G}(x)=1 $, for all $ x \in G $. A compact group $ G $ is always unimodular. 

It is known that $C_c(G/H)$ consists of all $P_{H}f$ functions, where $f\in C_c(G)$ and
\begin{align}
P_{H}f(xH)=\int_{H}f(xh)dh\quad(x\in G).\label{eq21}
\end{align}
Moreover, $ P_{H}:C_c(G) \rightarrow C_c(G/H) $ is a bounded linear operator which is not injective (see subsection 2.6 of \cite{3}). At this point, we recall that for a positive Radon measure $ \mu $ on the homogenous space $ G/H $ and $ x \in G $, the translation of $ \mu $ by $ x $ is the Radon measure $ \mu_{x} $ on $ G/H $ defined by $ \mu_{x}(E)= \mu(xE) $ for Borel set $ E \subseteq G $. $ \mu $ is called $ G $-invariant if $ \mu_{x}= \mu $ for all $ x \in G $.
A Radon measure $ \mu $ on $ G/H $ is said to be strongly quasi-invariant, if there exists a continuous function $ \lambda:G\times G/H\rightarrow (0, +\infty) $ which satisfies
\[d\mu_{x}(yH)=\lambda(x, yH)d\mu(yH).\]
If the function $ \lambda(x, .) $ reduces to a constant for each $ x \in G $, then $ \mu $ is called relatively invariant under $ G $. We consider a rho-function for the pair $ (G, H) $ as a non-negative locally integrable function $ \rho $ on $ G $ which satisfies $\rho(xh)=\triangle_{H}(h)\triangle_{G}(h)^{-1}\rho(x)$, for each $x\in G$ and $h\in H$. It is well known that $ (G, H) $ admits a rho-function which is continuous and everywhere strictly positive on $ G $. For every rho-function $ \rho $ there exists a regular Borel measure $ \mu $ on $ G/H $ such that
\begin{align}
\int_{G} f(x)\rho(x)dx=\int_{G/H}\int_{H} f(xh)dhd\mu(xH) \qquad (f\in C_{c}(G)).\label{eq22}
\end{align}
This equation is known as the quotient integral formula . The measure $\mu$ also satisfies
\[\frac{d\mu_{x}}{d\mu}(yH)=\frac{\rho(xy)}{\rho(y)}\qquad (x,y \in G).\]
Every strongly quasi-invariant measure on $ G/H $ arises from a rho-function in this
manner, and all of these measures are strongly equivalent (Proposition 2.54 and Theorem 2.56 of \cite{3}). Therefore, if $ \mu $ is a strongly quasi-invariant measure on $ G/H $, then the measures $ \mu_{x} $, $ x\in G $, are all mutually absolutely continuous. Trivially, $ G/H $ has a $ G- $ invariant Radon measure if and only if the constant function $ \rho(x)=1 $, $ x\in G $, is a rho-function for the pair $ (G, H) $.\\ 

If $ \mu $ is a strongly quasi invariant measure on $G/H$ which is associated with the rho$-$function $\rho$ for the pair $ (G, H) $, then the mapping $T_{H}: L^{1}(G)\rightarrow L^{1}(G/H)$ defined almost everywhere by
\begin{center}
$T_{H}f(xH)=\int_{H}\frac{f(xh)}{\rho(xh)}dh$ 
\end{center}
is a surjective bounded linear operator with $\parallel T_{H} \parallel \leq1$ (see \cite{11}) and also satisfies the generalized Mackey-Bruhat formula,
\begin{center}
$\int_{G}f(x)dx=\int_{G/H}T_{H}f(xH)d\mu(xH)\qquad (f\in L^{1}(G)),$ 
\end{center}
which is also known as the quotient integral formula.\\
\section{generalized quotient integral formula}
Throughout this section, we assume that $G$ is a locally compact group and $H$
is its closed subgroup. The space $G/H$ equip with quotient topology is considered as a homogeneous space that $G$ acts on it from the
left and $\pi_{H}:G\longrightarrow G/H$ denotes the canonical map. It is well known that $\pi_{H}$ is open, surjective and continuous. This can be generalized to homogeneous spaces as follows: 
\begin{proposition}\label{pr31}
Let $H$ be a closed subgroup of locally compact group $G$, and $L$ be a closed subgroup of $H$. The map $\pi_{L, H}:G/L\to G/H$ defined by $\pi_{L, H}(xL)=xH$ is open, surjective and continuous. 
\end{proposition}
\begin{proof}
Trivially the map $\pi_{L, H}$ is well defined.
We have $\pi_{H}=\pi_{L, H}\circ \pi_{L}$, where $\pi_{L}:G\to G/L$, $\pi_{H}:G\to G/H$ are the canonical quotient maps. By using Lemma $ 3.1.9 $ of \cite{11}, since $\pi_{L}$ is an open and surjective map and $\pi_{H}$ is a continuous map, we have $\pi_{L, H}$ is continuous. Also, $ \pi_{L, H}$ is an open map because for each open set $U$ in $G/L$
\begin{align*}
\pi_{L, H}(U)&=\{\pi_{L, H}(xL)\ |\ xL\in U\}\\
&=\{xH\ |\ xL\in U\}\\
&=\{\pi_{H}(x)\ |\ xL\in U\}\\
&=\pi_{H}\{x\in G\ |\ xL\in U\}\\
&=\pi_{H}\{x\in G\ |\ \pi_{L}(x)\in U\}\\
&=\pi_{H}(\pi_{L}^{-1}(U)),
\end{align*}
meanwhile $\pi_{L}$ is a continuous map and $\pi_{H}$ is an open map.\\
\end{proof}
We should clear that $ H $ is not normal subgroup necessarily, so $ G/H $ does not possess group stucture but it will be a locally compact Hausdorff space and also it is worthwhile to note that if $ G/L $ is locally compact (respectively, compact), then $ G/H $ is also locally compact (respectively, compact).
Suppose that $\Delta_H\big|_L=\Delta_L$, where $\Delta_H$ and $\Delta_L$ are the modular functions of $H$ and $L$, respectively. In this case, the quotient integral formula guarantees the existence of a unique (up to a constant) $H$-invariant measure $\eta$ on $H/L$. Now, we define the Radon transform $R_{L, H}:C_c(G/L)\to C_c(G/H)$ by
\begin{align}
R_{L, H}f(xH)=\int_{H/L}f(xhL)d\eta(hL)\qquad (f\in C_{c}(G/L)),\label{eq31}
\end{align}
which is well defined. Note that if $L $ is the trivial subgroup of $ H $ then clearly $ R_{L, H}$ is $P_{H}$ (see \eqref{eq21}).
\begin{proposition}\label{pr32}
If $f\in C_c(G/L)$, then $ R_{L, H}f \in C_c(G/H) $ and $R_{L, H}[(\phi\circ \pi_{L, H})\cdot f]=\phi\cdot R_{L, H}f$, for all $ \phi \in C_c(G/H) $.
\end{proposition}
\begin{proof}
For any $f\in C_c(G/L)$ there exists $\psi \in C_c(G)$ such that $ P_{L}(\psi)=f $. Since $ P_{H}=R_{L, H}\circ P_{L} $, $R_{L, H}(f)=R_{L, H}(P_{L}(\psi))=P_{H}(\psi)$ belongs to $ C_c(G/H) $. So the map $ R_{L, H}:C_c(G/L)\to C_c(G/H)$ is well defined. Moreover, we have

\begin{align*}
R_{L, H}[(\phi\circ \pi_{L, H})\cdot f](xH)&=\int_{H/L}[\phi\circ \pi_{L, H})\cdot f](xhL)d\eta (hL)\\
&=\int_{H/L}(\phi\circ \pi_{L, H})(xhL)\cdot f(xhL)d\eta(hL)\\
&=\int_{H/L} \phi(xH)\cdot f(xhL)d\eta (hL)\\
&=\phi(xH)\int_{H/L}f(xhL)d\eta(hL)\\
&=\phi(xH)R_{L, H}f(xH)\\&=(\phi\cdot R_{L, H}f)(xH),
\end{align*}
so $R_{L, H}[(\phi\circ \pi_{L, H})\cdot f]=\phi\cdot R_{L, H}f$.\\ 
\end{proof}
The Lemma $ 3.3 $, Proposition  $ 3.4 $ and the Proposition $ 3.5 $ are generalizations of the ones given in \cite{3}.
\begin{lemma}\label{le3}
Let $\pi_{L, H}:G/L\to G/H$ be the canonical map. If $E$ is a compact subset of $G/H$, then there exists a compact subset $K$ of $ G/L$ such that $\pi_{L, H}(K)=E$.
\end{lemma}
\begin{proof}
Let $E$ be a compact subset of $G/H$. By Lemma 2.46 in \cite{3}, there exists a compact subset $ C $ in $ G $ with $ \pi_{H}(C)=E $. Since $ \pi_{L,H}\circ\pi_{L}=\pi_{H} $, $ K:=\pi_{L}(C) $ does the job.
\end{proof}

\begin{proposition}\label{pr4}
If $ \phi \in C_c(G/H) $, there exists $ f \in C_c(G/L) $ such
that $ R_{L, H}f=\phi $ and $\pi_{L, H}(supp\ f)=supp\ \phi $. Also any element of $C^{+}_c(G/H)$ is in the form $ R_{L, H}f $, for some $f\in C^{+}_c(G/L)$.
\end{proposition}
\begin{proof}
Let $\phi \in C_c(G/H)$. By Proposition 2.48 in \cite{3}, there exists $\psi \in C_c(G) $ such that $ P_{H}(\psi)=\phi $ and $ supp \pi_{H} (\psi) = supp \phi $. Set $ f:=P_{L}(\psi) $. Since $ P_{H}=R_{L, H}\circ P_{L} $, $R_{L, H}(f)=R_{L, H}(P_{L}(\psi))=P_{H}(\psi)=\phi $.
Clearly if $ \phi \geq 0 $, choose  $  \psi \geq 0  $, So $ f\geq 0 $.
\end{proof}
The following theorems are a generalization of the quotient integral formula and illustrate how two Radon measures on homogenuous spaces are related. 

\begin{theorem}
\label{th5}
Suppose $G$ is a locally compact group, $H$ is a closed subgroup of $G$ and $L$ is a closed subgroup of $H$ such that $\Delta_H\big|_L=\Delta_L$. There exist $G$-invariant Radon measures $\mu$ and $\nu$ on $G/H$ and $G/L$, respectively, if and only if $\Delta_G\big|_H=\Delta_H$. In this case, these two measures are unique up to a constant factor, and if this factor is suitably chosen, we have 
\begin{align*}
\int_{G/L}f(xL)d\nu(xL)\
&=\int_{G/H}R_{L, H}f(xH)d\mu(xH)\\
&=\int_{G/H}\int_{H/L}f(xhL)d\eta(hL)d\mu(xH)\qquad(f\in C_c(G/L)).
\end{align*}
\end{theorem}
\begin{proof}
Since $\Delta_G\big|_H=\Delta_H$ and $\Delta_H\big|_L=\Delta_L$, it follows that $\Delta_G\big|_L=\Delta_L$. Then by \eqref{eq22} there exists such measures. If $f\in C_c(G/L)$ there exists $\phi\in C_c(G)$ such that $P_{L}\phi=f$ and we have 
\begin{align}
&\int_G\phi(x)dx=\int_{G/H}\int_H\phi(xh)dhd\mu(xH)\label{eq32}\\
&\int_G\phi(x)dx=\int_{G/L}\int_L\phi(xl)dld\nu(xL)\label{eq33}\\
&\int_H\psi(h)dh=\int_{H/L}\int_L\psi(hl)dld\eta(hL),\quad(\psi\in C_c(H)).\label{eq34}
\end{align}
From \eqref{eq33},\eqref{eq32} and \eqref{eq34} we have 
\begin{align*}
\int_{G/L}f(xL)d\nu(xL)\
&=\int_{G/L}\int_L\phi(xl)dld\nu(xL)\\
&=\int_{G/H}\int_H\phi(xh)dhd\mu(xH).
\end{align*}
For a given $ x\in G $, define $ \psi: H\rightarrow C$ by $ \psi(h):=\varphi(xh) $. Trivially, $ \psi\in C_{c}(H) $ and

\begin{align*}
\int_{G/H}\int_H\phi(xh)dhd\mu(xH)\
&=\int_{G/H}\int_H\psi(h)dhd\mu(xH)\\
&=\int_{G/H}\int_{H/L}\int_{L}\psi(hl)dld\eta(hL)d\mu(xH)\\
&=\int_{G/H}\int_{H/L}\int_{L}\varphi(xhl)dld\eta(hL)d\mu(xH)\\
&=\int_{G/H}\int_{H/L}P_{L}\varphi(xhL)d\eta(hL)d\mu(xH)\\
&=\int_{G/H}\int_{H/L}f(xhL)d\eta(hL)d\mu(xH)\\
&=\int_{G/H}R_{L, H}f(xH)d\mu(xH).
\end{align*}
The converse is immediate by the quotient integral formula (Theorem 2.56 in \cite{3}). 
\end{proof}
Some manifolds that we met in a differential geometry class are homogeneous space: spheres, tori, Grassmannians, flag manifolds, Stiefel manifolds, etc are of this type. In the following example to elucidate the usefulness of Theorem \ref{th5}, we provide a Radon measure for a manifold by knowing it for another homogeneous space. 
\begin{example}
Recall that $G:=SL (2, \mathbf{R})$ acts transitively on the upper half plane $ \mathbf{H}^{+}:=\lbrace z\in \mathbf{C} \mid Im(z)>0 \rbrace$ via:
\begin{center}
\[\left(  {\begin{array}{cc}
   \alpha & \beta  \\
   \gamma & \eta  \\
    \end{array} } \right) z:=\frac{\alpha z+\beta}{\gamma z+\eta }.\]
\end{center}

One readily verifies that $H:=stab_{G}(i)=SO_{2}(\mathbf{R}) $; also, it can be checked that $ H $ acts transitively on the space of all lines $ \mathbf{P}^{1}(\mathbf{R})=\lbrace V\leq \mathbf{R}^{2} \mid dimV=1\rbrace $. Here, $L:=stab_{H}(<e_{1}>)=\lbrace \pm I\rbrace, $ so the maps
\begin{align*}
{\begin{array}{cc}
   G/H\rightarrow \mathbf{H}^{+}\qquad& \mathbf{H}^{+}\rightarrow G/H\qquad \qquad \\
    gH\mapsto gi,\qquad &   \qquad x+iy\mapsto \left({\begin{array}{cc}
   \sqrt{y} & x\sqrt{y}^{-1}  \\
   0 & \sqrt{y}^{-1}  \\
    \end{array} } \right),  \\
    \end{array} }
\end{align*}
and
\begin{align*}
{\begin{array}{cc}
   H/L\rightarrow \mathbf{P}^{1}(\mathbf{R})
\qquad& \mathbf{P}^{1}(\mathbf{R})\rightarrow H/L\qquad\qquad\qquad\qquad\qquad\qquad\qquad \\
   hL\mapsto h<e_{1}>,\qquad &   <\left( \cos \theta ,\sin\theta \right) >\mapsto \left({\begin{array}{cc}
   \cos \theta & -\sin \theta  \\
   \sin \theta & \cos \theta  \\
    \end{array} } \right) or \left({\begin{array}{cc}
   -\cos \theta & \sin \theta  \\
   \sin \theta & \cos \theta  \\
    \end{array} } \right).  \\
    \end{array} }\\
\end{align*}
are homeomorphisms. Since $G$, $H$ and $ L $ are unimodular, we conclude the existence of the Haar measures $ \mu $ and $ \nu $ on $G/H$  and $ H/L $, respectively. Therefore the linear map

\begin{align*}
C_{c}(G/L)\rightarrow \mathbf{C}, \qquad f\mapsto \int_{G/H}\int_{H/L}f(ghL)d\nu(hL)d\mu(gH)\qquad(f\in C_c(G/L)), 
\end{align*}
is a Radon measure on $ G/L $. We consider the homeomorphisms $ H/L\cong \mathbf{P}^{1}(\mathbf{R})  $ and $ G/H\cong \mathbf{H}^{+} $, then the functional
\begin{align*}
f\mapsto \frac{1}{2\pi}\int_{-\infty}^{+\infty}\int_{0}^{+\infty}\int_{0}^{2\pi}f\left( \left({\begin{array}{cc}
   \sqrt{y} & x\sqrt{y}^{-1}  \\
   0 & \sqrt{y}^{-1}  \\
    \end{array} } \right)\left( {\begin{array}{cc}
   \tan {\theta} \\
   0 \\
    \end{array} } \right)\right) d\theta \frac{dydx}{y^{2}}\qquad(f\in C_c(G/L)), 
\end{align*}
is a Radon measure with total mass $ \pi $ on $ \mathbf{P}^{1}(\mathbf{R}) $.

\end{example}

At this point, we suppose that $G/H$ does not possess a $G-$invariant Radon measure.
\begin{theorem}\label{th6}
Suppose $G$ is a locally compact group, $H$ is a closed subgroup of $G$ and $L$ is a closed subgroup of $H$. Let $\rho_{G, H}$ be an arbitrary rho$ - $function for the pair $(G, H)$ and $\Delta_G\big|_L=\Delta_H\big|_L=\Delta_L$ and also assume that $\nu$ and $\eta$ are $G$-invariant and $H$-invariant Radon measures on $G/L$ and $H/L$, respectively. There is a regular Borel measure $\mu$ on $G/H$ such that\\
\begin{align*}
\int_{G/H}R_{L, H}f(xH)d\mu(xH)\
&=\int_{G/L}\rho_{G, H}(xL)f(xL)d\nu(xL)\qquad(f\in C_c(G/L)). 
\end{align*}
\end{theorem}
\begin{proof}
By Proposition $1.15$ of \cite{6}, there exists a regular Borel measure $\mu$ on $G/H$ such that\\
\begin{align}
\int_{G/H}\int_{H}\phi(xh)dhd\mu(xH)=\int_{G}\phi(x)\rho_{G, H}(x)dx\qquad(\phi\in C_c(G)),\label{eq35}
\end{align}\\
and also by the quotient integral formula we have
\begin{align}
&\int_{G/L}\int_{L}\phi(xl)dld\nu(xL)=\int_{G}\phi(x)dx\qquad(\phi\in C_c(G)),\label{eq36}\\
&\int_{H/L}\int_{L}\phi(hl)dld\eta(hL)=\int_{H}\phi(h)dh\qquad(\phi\in C_c(H)).\label{eq37}
\end{align}\\

Let $f\in C_c(G/L)$. By Proposition $ 1.9 $ of \cite{6} there exists $\phi\in C_c(G)$ such that $ P_{L}\phi=f $,
where $P_{L}:C_c(G)\to C_c(G/L)$ is the following surjective bounded operator 
\begin{align}
P_{L}\phi(xL)=\int_{L}\phi(xl)dl\qquad(\phi\in C_c(G)).\label{eq38}
\end{align}
By using \eqref{eq31} and \eqref{eq38} we have 
\begin{align*}
\int_{G/H}R_{L, H}f(xH)d\mu(xH)\
&=\int_{G/H}\int_{H/L}f(xhL)d \eta(hL)d\mu(xH)\\
&=\int_{G/H}\int_{H/L}P_{L} \phi(xhL)d\eta(hL)d\mu(xH)\\
&=\int_{G/H}\int_{H/L}\int_{L}\phi(xhl)dld\eta(hL)d\mu(xH).
\end{align*}

Now, from \eqref{eq37}, \eqref{eq35} and \eqref{eq36} we have

\begin{align*}
\int_{G/H}\int_{H/L}\int_{L}\phi(xhl)dld\eta(hL)d\mu(xH)\
&=\int_{G/H}\int_{H}\phi(xh) dhd\mu(xH)\\
&=\int_{G}\phi(x)\rho_{G, H}(x)dx\\
&=\int_{G/L}\int_{L}\phi(xl)\cdot\rho_{G, H}(xl)dld\nu(xL).\\
&=\int_{G/L}\rho_{G, H}(x)\int_{L}\phi(xl)dld\nu(xL).\\
&=\int_{G/L}\rho_{G, H}(xL)\int_{L}\phi(xl)dld\nu(xL).\\
&=\int_{G/L}\rho_{G, H}(xL)f(xL)d\nu(xL).
\end{align*}

\end{proof}
The following corollay easily follows from the Theorem \ref{th6}.
\begin{corollary}\label{n3}
If $ f\in C_{c}(G/L) $ and $ R_{L, H}f=0 $ then $\int_{G/L}\rho_{G, H}(xL)f(xL)d\nu(xL)=0$.

\end{corollary}
\begin{example}
Suppose $G$ is a locally compact group, $H$ is a normal closed subgroup of $G$ and $L$ is a closed subgroup of $H$. Let $\rho_{G, H}$ be an arbitrary rho$ - $function for the pair $(G, H)$ and also let $\nu$ and $\eta$ be $G$-invariant and $H$-invariant Radon measures on $G/L$ and $H/L$, respectively. There is a strongly quasi-invariant measure $\mu$ on $G/H$ such that\\
\begin{align*}
\int_{G/H}R_{L, H}f(xH)d\mu(xH)\
&=\int_{G/L}\rho_{G, H}(xL)f(xL)d\nu(xL)\qquad(f\in C_c(G/L)). 
\end{align*}

\end{example}

\end{document}